\documentclass[a4paper,11pt,reqno]{amsart}
\usepackage{amsfonts}
\usepackage{amsmath}
\usepackage{amssymb}
\usepackage{bm}
\usepackage{graphicx}
\usepackage{xcolor}
\usepackage{enumitem}

\newtheorem{definition}{Definition}
\newtheorem{theorem}{Theorem}
\newtheorem{proposition}{Proposition}

\newtheorem{lemma}{Lemma}

\theoremstyle{remark}
\newtheorem{remark}{Remark}

\def\bm{\boldsymbol}
\def\ba{\begin{array}}
\def\ea{\end{array}}
\def\la{\label}
\def\p{\partial}

\def\Hcal{{\mathcal H}}

\def\Lcal{{\mathcal L}}
\def\Acal{{\mathcal A}}
\def\Mcal{{\mathcal M}}
\def\Ccal{{\mathcal C}}

\def\ol{\overline}
\def\f{\frac}
\def\Z{{\mathbb Z}}
\def\CC{C}

\begin{document}

\title[Tau function and moduli of meromorphic forms]{Tau function and moduli \\ of meromorphic forms on algebraic curves}

\author{D.~Korotkin}

\address{Department of Mathematics and Statistics, Concordia University,
1455 de Maisonneuve West, Montreal, H3G 1M8  Quebec,  Canada}
\email{dmitry.korotkin@concordia.ca}

\author{P.~Zograf}
\address{
St. Petersburg Department of Steklov Mathematical Institute, Fontanka 27, St. Petersburg 191023 Russia,
and\newline
Chebyshev Laboratory, St. Petersburg State University, 14th Line V.O. 29, Saint Petersburg 199178 Russia}
\email{zograf@pdmi.ras.ru}

\thanks{The main result of this paper, Theorem 1, was obtained at St.~Petersburg Department of Steklov Mathematical Institute under support of RSF grant 21-11-00040. In addition to that, the work of DK was supported by NSERC grant RGPIN-2020-06816, and the work of PZ was supported by Ministry of Science and Higher Education of the Russian Federation, agreement № 075–15–2022–289.}

\begin{abstract}

We study the moduli space of meromorphic 1-forms on complex algebraic curves having at most simple poles with fixed nonzero residues. We interpret the  Bergman tau function on this moduli space as a section of a line bundle and study its asymptotic behavior near the boundary and the locus of forms with non-simple zeros. As an application, we decompose the projection of this locus to the moduli space of curves into a linear combination of standard generators of the rational Picard group with explicit coefficients that depend on the residues.

\end{abstract}
\maketitle



The present note is a continuation of a series of papers by the authors devoted to applications of tau functions to geometry of moduli spaces and can be considered as an extension of \cite{MRL}, where the case of holomorphic 1-forms (Abelian differentials of the first kind) was treated. The case of meromorphic 1-forms with simple poles (Abelian differentials of the third kind) basically follows the same lines, but involves some new technicalities like the dependence on relations between residues at poles and the appearance of tautological $\psi$-classes similar to \cite{BK}.

We begin with introducing the necessary notation.
Denote by  $\Hcal_{g,n}$  the space of equivalence classes of pairs $(C,v)$,
where $C$ is a smooth complex curve of genus $g$ with $n$ labeled marked points $p_1,\ldots, p_n$, and $v$ is a meromorphic 1-form on $C$ with at most simple poles at $p_i$ and no other poles; by Riemann-Roch, $\dim \Hcal_{g,n}=4g-4+2n$. The forgetful map  
$\Hcal_{g,n}\to\Mcal_{g,n}$ is a vector bundle (in the orbifold sense) with fiber $H^0(C,K_C(p_1+\ldots+p_n))$ over $(C,v)$
of dimension $g-1+n$, where $K_C$ is the canonical class of $C$. The bundle $\Hcal_{g,n}$ naturally extends to the Deligne-Mumford compactification $\ol{\Mcal}_{g,n}$ of $\Mcal_{g,n}$. More precisely, 
$\ol{\Hcal}_{g,n}=\pi_*(\omega_{g,n}(p_1+\ldots+p_n))$, where $\omega_{g,n}$ is the relative dualizing sheaf of the universal curve $\pi:\ol{\Ccal}_{g,n}\to\ol{\Mcal}_{g,n}$, and $p_1,\ldots, p_n$ are the structure sections of $\pi$.

A classical result that can be traced back to Riemann claims that for any $(C;p_1,\ldots, p_n)\in{\Mcal}_{g,n}$ and any
$(r_1,\ldots,r_n)\in\mathbb{C}^n$ there exists a meromorphic 1-form $v$ on $C$ having at most simple poles at the points 
$p_1,\ldots, p_n\in C$ with residues $r_1,\ldots,r_n\in\mathbb{C}$ and smooth elsewhere if and only if 
$r_1+\ldots+r_n=0$. Such a 1-form is defined uniquely up to an arbitrary holomorphic 1-form on $C$. 
This result naturally extends to the case of stable curves (where the 1-forms are allowed to have simple poles at the nodes,
see below) and has two important consequences:

(1) The {\em residue map} ${\rm res}: \ol{\Hcal}_{g,n}\longrightarrow\mathbb{C}^n,\;(C,v)\mapsto (r_1,\ldots,r_n)$, where $r_i={\rm res}\left|_{p_i}v\right.,\;i=1,\ldots,n$, fits into the exact sequence
\[
\ol{\Hcal}_{g,n}\stackrel{\rm res}{\longrightarrow}\mathbb{C}^n\stackrel{\sum r_i}{\longrightarrow}\mathbb{C}.
\]

(2) Any fiber $\ol{\Hcal}_{g,n}[{\bm r}]={\rm res}^{-1}(r_1,\ldots,r_n)$ of the residue map, where 
${\bm r}=(r_1,\ldots,r_n)$, is an affine bundle on $\ol{\Mcal}_{g,n}$ of dimension $g$. For smooth curves, the fibers of the projection ${\Hcal}_{g,n}[{\bm r}]\longrightarrow{\Mcal}_{g,n}$ are affine spaces over the vector spaces of holomorphic 1-forms on $C$, and a minor adjustment is needed for nodal curves.

By a result of Grothendieck \cite{Gr} (cf. also \cite{Fu}, p. 69), the projection
\[
{\rm pr}:\ol{\Hcal}_{g,n}[{\bm r}]\longrightarrow\ol{\Mcal}_{g,n}
\]
induces an isomorphism between the rational Picard groups ${\rm Pic}(\ol{\Hcal}_{g,n}[{\bm r}])\otimes\mathbb{Q}$ and
${\rm Pic}(\ol{\Mcal}_{g,n})\otimes\mathbb{Q}$.

Let us recall the standard set of generators of 
${\rm Pic}(\ol{\Mcal}_{g,n})\otimes\mathbb{Q}$ (cf. \cite{AC}). It consists of:
\begin{itemize}
\item the Hodge class $\lambda=c_1(\pi_*\omega_{g,n})$,
\item the classes $\psi_i=c_1(\Lcal_i)$ of the tautological line bundles $\Lcal_i\to\ol{\Mcal}_{g,n}\,, \linebreak i=1,\ldots,n$,
\item the class $\delta_{irr}$ of the boundary divisor parameterizing irreducible nodal curves, and
\item the classes $\delta_{j,I}$ of the boundary divisors parameterizing reducible curves with components of genera $j$ and $g-j$ containing $k$ and $n-k$ marked points $p_i$ respectively
(here $j=0,\ldots,[g/2]$, and the multi-index $I=\{i_1,\ldots,i_k\}$ runs over all subsets of $[n]=\{1,\ldots,n\}$;
the unstable cases $j=0,\;k=0,1$ are excluded).\footnote{Slightly abusing notation, we denote a divisor and its class 
in ${\rm Pic}(\ol{\Mcal}_{g,n})\otimes\mathbb{Q}$ by the same symbol.}
\end{itemize}

The moduli space ${\Hcal}_{g,n}$ admits a natural stratification according to the multiplicities of zeros and poles. Each stratum has a linear structure given by the
{\em period coordinates}, where  the transition maps are linear with integer coefficients. 
In the case of the principal stratum (i.e. the locus of 1-forms with only simple zeros)
these coordinates are the integrals of $v$ over a set of generators of the relative
homology group $H_1(C\setminus \{z_i\}_{i=1}^{2g-2+n}\cup\{p_j\}_{j=1}^n)$ 
that includes the periods of $v$ over a basis of canonical cycles on $C$, the integrals of $v$ between its zeros $z_i$, and the residues of $v$ at the poles $p_j$.

The principal stratum is open and dense in 
$\ol{\Hcal}_{g,n}$, and its complement is a divisor whose class we denote by 
$\delta_{deg}$ (the subscript $deg$ stands for {\em degenerate} as opposed to generic).
The moduli space $\ol{\Hcal}_{g,n}$ fibers into affine spaces parameterizing 1-forms with fixed residues ${\bm r}=(r_1,\ldots,r_n)$, where $r_1+\ldots+r_n=0$. By linearity, the intersection $\delta_{deg}\cdot\ol{\Hcal}_{g,n}[{\bm r}]$ has codimension 1 in
$\ol{\Hcal}_{g,n}[{\bm r}]$ and is locally described by vanishing of the integral of $v$ along a path connecting two coalescing zeros of $v$, cf. \cite{KonZor}, Section 5.2, for a relevant discussion. 

The main objective of this note is to decompose the pushforward class
\[
\delta_{deg}[{\bm r}]={\rm pr}_*\left(\delta_{deg}\cdot\ol{\Hcal}_{g,n}[{\bm r}]\right)
\]
into a linear combination of the standard generators of ${\rm Pic}(\ol{\Mcal}_{g,n})\otimes\mathbb{Q}$, see Theorem 1.
Rather unexpectedly, the result depends on partial sums of the residues being zero or not.

Now we remind  the basic facts about the tau function, cf. \cite{JDG}.
Choose a canonical basis $\{a_1,b_1,\ldots , a_g,b_g\}$  in $H_1(C,\Z)$. It defines a normalized basis $\{v_1,\dots,v_g\}$ in the space of holomorphic 1-forms on $C$ by the condition $\int_{a_i}v_j=\delta_{ij},\;i,j=1,\ldots,g$.  Denote by $\Omega$ the period matrix
$\Omega_{ij}=\int_{b_i}v_j$ of $C$.

Consider the following multi-valued $g(1-g)/2$-differential on $C$:
\begin{align}
c(x)=\f{1}{W(x)}\left(\sum_{i=1}^g v_i(x)\f{\p}{\p w_i}\right)^g\Theta(w,\Omega)\Big|_{w=K^x}
\label{defCc}
\end{align}
(see  \cite{Fay92}). Here $\Theta$ is the theta function on $C$, 
$K^x$ is the vector of Riemann constants with respect to the base point $x\in C$ defined by
$$
K^x_j=\frac{1+\Omega_{jj}}{2}-\sum_{i=1, \; i\neq j}^g \int_{y\in a_i} \left(v_i(y)\int_x^y v_j\right)\;
$$
 (see \cite{Fay92}, p.~6), and $W(x)$ is the Wronskian of holomorphic 1-forms defined as the usual Wronskian 
$$
W(x)=W(v_1,\dots,v_g)= W(f_1(\xi),\dots,f_g(\xi)) (d\xi)^{g(g+1)/2}
$$
of  $g$ functions $f_1(\xi),\ldots,f_g(\xi)$, where $\xi$ is a local coordinate near $x\in C$ and $v_i(\xi)=f_i(\xi)d\xi,
\;i=1,\ldots,g$.

The differential $v$ gives rise to a natural set of the so-called {\em distinguished} local parameters on $C$ (see \cite{JDG,CMP}) defined as follows:

\begin{itemize}
\item near any regular point $y$ (that is, $y\not\in(v)$, where $(v)$ is the divisor of $v$ on $C$), the distinguished local parameter is $\zeta(x)=\int_{x}^y v$.
\item near a simple zero $z_i$ of $v$ the distinguished local parameter is
\begin{align}
\xi_i(x)=\left(\int_{z_i}^x v\right)^{1/2},\quad i=1,\ldots, 2g-2+n\;
\label{distxiab}
\end{align}
(here we assume that all zeros of $v$ are simple and numbered), and
\item near a simple pole $p_j$ of $v$ the distinguished local parameter is
\begin{align}
\zeta_j(x)=\exp\left(\f{2\pi \sqrt{-1}}{r_j}\int_{z_1}^x v\right),\quad j=1,\dots, n\,,
\label{distziab}
\end{align}
where the contour of integration in (\ref{distziab}) connects the first zero $z_1$ with a neighborhood of the pole $p_j$.
\end{itemize}

The Schottky--Klein {\em prime form} 
$E(x,y)$ is a (normalized) holomorphic section of the line bundle with the divisor
$\{x=y\}$ on $C\times C$ (see \cite{Fay92}).  
It is explicitly given by the formula 
\begin{equation}
E(x,y)=\frac{\Theta[*](\Acal(x)-\Acal(y))}{\sqrt{h(x)}\sqrt{h(y)}}\,,
\label{defE}
\end{equation}
where $\Theta [*]$ the theta function  
with an arbitrary half-integer characteristic $[*]$, $\Acal(x)$ is the Abel map, and
$h(x)$ is a holomorphic 1-form on $\CC$ with only double zeros defined by
\[
h(x)=\sum_{j=1}^g \frac{\partial}{\partial w_j}\Theta[*](w)\Big|_{w=0} v_j(x)\,.
\]

The values of the prime forms  $E(x,z_i)$, $E(x,p_j)$,  $E(z_i,z_k)$,  $E(p_j,p_l)$ and  $E(z_i,p_j)$, where one of or both arguments coincide with the points of
divisor $(v)$, can be computed in terms of the distinguished local coordinates (\ref{distxiab}), (\ref{distziab}).
Namely, 
\begin{align}
&E(x,z_i)= E(x,y) \sqrt{d\xi_i(y)}|_{y=z_i}\;,\nonumber\\
&E(x,p_j)= E(x,y) \sqrt{d\zeta_j(y)}|_{y=p_j}\;,\nonumber\\
&E(z_i,z_j)=E(x,y) \sqrt{d\xi_i(x)d\xi_j(y)}|_{x=z_i,\;
y=z_j}\;,\nonumber\\
&E(p_i,p_j)=E(x,y) \sqrt{d\zeta_i(x)d\zeta_j(y)}|_{x=p_i,\;
y=p_j}\;,\nonumber\\
&E(p_i,z_j)=E(x,y) \sqrt{d\zeta_i(x)d\xi_j(y)}|_{x=p_i,\;
y=z_j}\;.
\end{align}

Since all zeros of $v$ are assumed to be simple, we can choose for this computation a fundamental domain of $C$ such
that $\Acal_x((v))+2K^x=0$, where $\Acal$ is the Abel map (see \cite{JDG}, Lemma 6).

Now we are in a position to introduce the following 
\begin{definition}
The tau function is given by the formula
\begin{align}
&\tau(C,v)=\nonumber\\
&c^{16}(x)\left(\frac{v(x)\prod_{i=1}^n E(x,p_i)}{\prod_{i=1}^{2g-2+n} E(x,z_i)} \right)^{8(g-1)}
\left(\f{\prod_{i<j}E(z_i,z_j)\prod_{i<j} E(p_i,p_j)}{\prod_{i=1}^n \prod_{j=1}^{ 2g-2+n}E(p_i,z_j) }\right)^{4}\,.
\label{tauab3}
\end{align}
\end{definition}
Here and below we adopt the convention that the product over $i<j$ means that it is taken over all pairs of indices $i$ and $j$ in their respective range (i.e. from $1$ to $2g-2+n $ in the case of $E(z_i,z_j)$ and from $1$ to $n$ in the case of $E(p_i,p_j)$.

\begin{lemma}
The tau function defined by (\ref{tauab3}) is independent of $x\in \CC$.
\end{lemma}
\begin{proof}
For convenience, we write the divisor of $v$ as
\begin{equation}
\label{defv}
(v)=\sum_{i=1}^{2g-2+n} z_i-\sum_{i=1}^n p_i= \sum_{j=1}^{N} k_j x_j\,,
\end{equation}
where $N=2g-2+2n$ and $k_i=\pm 1$ ($+1$ for $z_i$'s and $-1$ for $p_j$'s).
Then the expression (\ref{tauab3}) can be rewritten as follows:
\begin{equation}
\tau=c^{16}(x)\left(\frac{v(x)}{\prod_{i=1}^N E^{k_i}(x,x_i)}\right)^{8(g-1)} \prod_{i<j} E^{4 k_i k_j}(x_i, x_j)
\label{taugen}
\end{equation}
(recall that it holds when the fundamental polygon is chosen such that $\Acal_x((v))+2 K^x =0$). 
Note that \eqref{taugen} formally coincides with (3.24) and (3.25) of \cite{JDG} except that some of $k_j$'s can now be negative.
 Observe that the right-hand side of  (\ref{taugen}) is non-singular and non-vanishing on $\CC$: all zeros and poles of $v$ are cancelled by the corresponding poles or zeros of $E(x,x_j)$. Moreover, the tensor weight of $\tau$ with respect to $x$ is $0$
since
\begin{itemize} 
\item $c^{16}(x)$ is a  differential of weight $8g(1-g)$, 
\item $v(x)^{8(g-1)}$  is a differential of weight $8(g-1)$, 
\item the product of prime forms in the denominator of \eqref{taugen} is a differential of weight $(-1/2)(8(g-1)) \sum_{j=1}^N k_j = -8(g-1)^2$. 
\end{itemize} 
 It remains to verify that $\tau$, as a function of $x$, is single-valued on $\CC$. When $x$ goes around any of $a$-cycles $a_j$, both $c(x)$ and the prime forms $E(x,x_j)$ remain invariant, possibly up to a sign, thus
 $\tau$ does not change. To compute the change of $\tau$ around a cycle $b_j$ we recall that the prime form
 $E(x,y)$ transforms as follows \cite{Fay92}:
 $$E(b_j(x),y)=\pm  e^{-\pi \sqrt{-1} \Omega_{jj} -2\pi \sqrt{-1} (\Acal_j(x)-\Acal_j(y))} E(x,y)\,,$$
 so that
 $$E(b_j(x),x_k)=\pm  e^{-\pi \sqrt{-1} \Omega_{jj} +2\pi \sqrt{-1} (\Acal_x(x_k))_j} E(x,y)\,,$$
 while for the factor $c(x)$ we have
 $$c(b_j(x))=e^{-\pi \sqrt{-1} (g-1)^2 \Omega_{jj} - 2\pi \sqrt{-1} (g-1) K_j^x} c(x)$$
 (here $b_j(x)$ means that the argument $x$ goes along $b_j$ once).
 Thus, we have
 \begin{align*}
 \log\frac{\tau(b_j(x))}{\tau(x)}&\\
 &= 16(-\pi \sqrt{-1} (g-1)^2\Omega_{jj}-2\pi \sqrt{-1} (g-1)K_j^x)\\
&-8(g-1)(-\pi \sqrt{-1} \Omega_{jj} \sum_{i=1}^N k_i +2 \pi \sqrt{-1} \sum_{i=1}^N k_i (\Acal_x(x_i))_j)=0
 \end{align*}
 since $\sum_{i=1}^N k_i =2g-2$ and $\sum_{i=1}^N \Acal_x(x_i) =-2 K_\alpha^x$.
 Therefore, $\tau(x)$ is a holomorphic function on  $\CC$ and as such is $x$-independent. 
\end{proof}
  
Let us make a few  more comments about definition (\ref{tauab3})
\begin{itemize}

\item
The tau function on the space of holomorphic 1-forms \cite{JDG,MRL} was originally defined as a covariant constant section of the trivial line bundle on the corresponding moduli space with respect to a  certain flat connection. The equations for the tau function were then explicitly integrated that led to expression 
 (\ref{tauab3}) with $n=0$. In the case of the moduli space of meromorphic 1-forms, a derivation of defining equations for the tau function was just outlined in \cite{CMP}. In order to keep exposition as self-contained as possible, we decided not to use them in the present paper and to derive all the required properties of the tau function starting from the explicit formula 
 (\ref{tauab3}). This would also provide alternative proofs to the results of \cite{JDG,MRL}.

\item
Expression (\ref{tauab3}) is a scalar that seemingly depend on the choice of distinguished local coordinates near $z_i$ and $p_j$. Notice that the parameters $\xi_i$ in
(\ref{distziab}) are defined up to a sign, but since the powers of all prime forms in (\ref{tauab3})
are multiples of $4$, the formula for $\tau$ is invariant with respect to the choice of these signs.

\item
In turn, the parameters $\zeta_j$ given by (\ref{distziab}) depend on the choice of the 
first zero $z_1$ and the choice of integration paths connecting $z_1$ with the poles $p_j$.
This dependence propagates to the tau function (\ref{tauab3}), and
an easy computation shows that the total power of the differential $d \zeta_i$ in (\ref{tauab3}) is equal to $-2$. Namely,
it enters $4(g-1)$ times to the power of $E(x,p_i)$ and $2(n-1)$ times to the power of
$E(p_i,p_j)$ in the numerator of $\tau(C,v)$, and $-2(2g-2+n)$ times to the power of $E(p_i,z_j)$ in the denominator (taken with the minus sign), summing up to $-2$. Thus, the product
\begin{align}
\tau(C,v)\left(\prod_{i=1}^n d \zeta_i(p_i)\right)^{2}
\label{defTcal}
\end{align}
does not depend on the choice of local coordinate $\zeta_i$ near $p_i, \; i=1,\ldots,n,$ given by  (\ref{distziab}).
\end{itemize}

Yet another expression for the tau function uses the bidifferential
\begin{equation}
\sigma(x,y)=\left(\frac{c(x)}{c(y)}\right)^{1/(1-g)}\,,
\la{sigdef}
\end{equation}
where $\sigma(x,y)$ is a multivalued  differential of weight $-g/2$ in $x$ and weight $g/2$ in $y$, cf.  \cite{Fay92}.
We have
\begin{proposition}
Up to a multiplicative constant, the tau function (\ref{taugen}) is given by the formula
\begin{equation}
\label{tauKoko}
\tau(C,v)= c^{16}(x) \left(\prod_{j=1}^N  \sigma^{4 k_j} (x,x_j)\right) \left(\frac{v(x)}{  \prod_{j=1}^N E^{k_j}(x,x_j)}\right)^{4(g-1)}\,,
\end{equation}
where
$$
 \sigma(x,x_l)=\sigma(x,y)(d \xi_k(y))^{-g/2}\Big|_{y=x_l}\,,
 $$
and $\xi_l$ is a distinguished local coordinate near $x_l$.
\end{proposition}
\begin{proof}
To begin with,  the expression (\ref{tauKoko}) was obtained in \cite{Koko}, formula (1.8), when all $k_i=1$, and in a preliminary version of \cite{JDG}, arXiv:math/0405042v1, formulas (4.4) and (4.5), when  all $k_i>0$. Here we give a proof for arbitrary $k_i$. 
 Using the independence of (\ref{taugen}) on $x$, we can express the ratio $c(x)/c(y)$ via the prime forms like follows:
 $$
 \sigma(x,y)=\left(\frac{v(x)}{v(y)}\prod_{i=1}^N   \frac{E^{k_i}(y, x_i)}{ E^{k_i}(x, x_i)} \right)^{1/2}\,.
 $$
From here we get
\begin{align}
\sigma(x,x_l)&= \left(\frac{v(x)}{\prod_{i=1}^N E^{k_i}(x,x_i)}\right)^{1/2} \prod_{i\neq l} E^{1/2}(x_l,x_i)\nonumber\\
&\times \left(\frac{E^{k_l}(y,x_l) (d\xi_l(y))^{1+k_l/2}}{v(y)}\Big|_{y=x_l}\right)^{1/2}
\label{sixxl}
\end{align}
 The last term in the product (\ref{sixxl}) is constant since $\xi_l$ is a distinguished local coordinate near 
 $x_l$.  
 Consider first the case $k_l=1$, i.e. $x_l=z_i$ is a simple zero of {v}. Then 
 $v(y)=2\xi_i(y)d\xi_i(y)$ and $E(y,x_i)\sim \xi_i(y) (d\xi_i(y))^{-1/2}$ so that
 $$
 \frac{E(y,z_i) (d\xi_i(y))^{3/2}}{v(y)}\Big|_{y=z_i}=\frac{1}{2}\,.
 $$
 Another case relevant for this paper is $k_l=-1$, i.e. $x_l=p_j$ is a simple pole of $v$. Then 
 $v(y)=\frac{r_j}{2\pi \sqrt{-1}}\frac{d\zeta_j(y)}{\zeta_j(y)}$,
 $E(z_j,y)\sim \zeta_j(y) (d\zeta_j(y))^{-1/2}$ 
 and
 $$
 \frac{ (d\zeta_j(y))^{1/2}}{v(y)E(y,\zeta_j)}\Big|_{y=\zeta_j}=\frac{2\pi \sqrt{-1}}{r_j}\;.
 $$
 Therefore,
 $$
 \prod_{j=1}^N \sigma^{4 k_j}(x,x_j)=const \left( \frac{v(x)}{\prod_{j=1}^N E^{k_j} (x,x_j)}\right)^{4g-4} \prod_{i<j} E^{4 k_i k_j}(x_i,x_j)\,,
 $$
 which, being substituted into (\ref{tauKoko}), gives (\ref{taugen}).
 \end{proof}

Now we are in a position to formulate 
\begin{proposition}
\label{seclL}
Fix $r_1,\ldots,r_n\in\mathbb{C}$ such that $\sum_{i=1}^nr_i=0$ and $r_i\neq 0,\;i=1,\ldots,n$.
Then the  tau function $\tau (C,v)$ is a  holomorphic section of the line  bundle 
$(\det\pi_*\omega_{g,n})^{\otimes 24}
\otimes\bigotimes_{i=1}^n \Lcal_i^{\otimes 2}$,
nowhere vanishing over the principal stratum $\ol{\Hcal}_{g,n}[{\bm r}]\cap\left({\Hcal}_{g,n}\setminus\delta_{deg}\right)$
in the moduli space $\ol{\Hcal}_{g,n}[{\bm r}]$.
\end{proposition}
A proof of this Proposition  based on transformation properties of the prime form and the differential $c(x)$ \cite{Fay92} is given in Appendix \ref{symplectic}.

Now let us compute the divisor of the tau function $\tau(C,v)$ on $\ol{\Hcal}_{g,n}[{\bm r}]$, or, equivalently, the
asymptotics of $\tau(C,v)$ near $\delta_{deg}\cdot\ol{\Hcal}_{g,n}[{\bm r}],\;{\rm pr}^*\delta _{irr}$,  
and  ${\rm pr}^*\delta_{j,I}$.    

\textbf{Asymptotics of $\tau$ near $\delta_{deg}\cdot\ol{\Hcal}_{g,n}[{\bm r}]$.} 
We begin with computing the asymptotics of the tau function  $\tau(C,v)$
near the divisor $\delta_{deg}$ in the total moduli space $\ol{\Hcal}_{g,n}$. Let $z_1,\;z_2$ be any two simple zeros of $v$
on $C$ that coalesce to a double zero. 

\begin{lemma}
The tau function has the following asymptotics near $\delta_{deg}$:
\begin{equation}
\label{deg}
\tau(C,v)={const}\cdot t + o(t)\quad {\textrm as}\; t\to 0
\end{equation}
with some $const\neq 0$, where $t$ is a transversal local coordinate on $\ol{\Hcal}_{g,n}$ near $\delta_{deg}$.
\end{lemma}
\begin{proof}
Consider a domain $D$ on $C$ with local coordinate $s$ containing both $z_1$ and $z_2$;
we assume that 
$$
v=  s (s-\varepsilon) ds\;.
$$
Then $\delta_{deg}$ is locally defined by equation $\varepsilon=0$, and the local transversal coordinate $t$ near $\delta_{deg}$ can be chosen to be 
$t=\varepsilon^2$ (since $z_1$ and $z_2$ are interchangeable; notice that the period coordinate $\int_{z_1}^{z_2} v = -\varepsilon^3/6$).
The distinguished local coordinates near $z_1$ and $z_2$ are given by $\xi_1(s)= (\int_0^s v)^{1/2}$ and $\xi_2(s)= (\int_{\varepsilon}^s v)^{1/2}$ respectively, so that
\begin{equation}
\frac{d\xi_1(s)}{ds}\Big|_{s=0}=\left(-\frac{\varepsilon}{2}\right)^{1/2}\;,\qquad
 \frac{d\xi_2(s)}{ds}\Big|_{s=\varepsilon} =\left(\frac{\varepsilon}{2}\right)^{1/2}\;.
\label{diska}\end{equation}

The following terms in (\ref{tauab3}) do not have finite nonzero limit as $z_1\to z_2$.

First, this is the theta function in the numerator of the prime form
$E^4(z_1,z_2)$, which contributes the term $\varepsilon^4$ to the  asymptotics of $\tau$ as $\varepsilon\to 0$ (since the theta-function with 
non-singular odd characteristics has simple zero at the origin).

Second, these are derivatives $\frac{d \xi_1}{ds}(z_1)$ and $\frac{d \xi_2}{ds}(z_2)$, each of which  
enter with power $-2$ (that appears in the same way as the power of $d\zeta_i$ 
computed before (\ref{defTcal})).  Therefore, these derivatives contribute the power $\varepsilon^{-2}$ to the asymptotics of $\tau$ as $\varepsilon\to 0$, and in total we get $\tau\sim \varepsilon^2$ as $\varepsilon\to 0$ which implies (\ref{deg}) since $t=\varepsilon^2$.
\end{proof}

\textbf{Asymptotics of $\tau$ near ${\rm pr}^*\delta _{irr}$.} This asymptotics is also computed  using only the explicit formula for the tau function. As the curve $C$ degenerates to an irreducible curve  with a node, 
it acquires two new  
simple poles at the preimages of the node with opposite residues, while its genus drops by 1. 

\begin{lemma}
The tau function  (\ref{tauab3}) has the following asymptotics near ${\rm pr}^*\delta _{irr}$:
\begin{align}\label{irr}
\tau(C,v)=const\cdot t^2+o(t^2)\quad {\rm as}\; t\to 0
\end{align}
with some $const\neq 0$, where $t$ is a transversal local coordinate on $\ol{\Hcal}_{g,n}$ near ${\rm pr}^*\delta _{irr}$.
\end{lemma}
\begin{proof} The expression (\ref{tauab3}) for the tau function is written under such choice of the fundamental polygon  that $\Acal_x((v))+2K^x=0$; however, this is not the choice suitable for analyzing the asymptotics near ${\rm pr}^*\delta _{irr}$. Instead, we are going to choose the fundamental domain such that
\begin{equation}
\Acal_x((v))+2K^x-\Omega e_g=0
\label{newchoice}
\end{equation}
where $e_g={\rm ^t}\!(0,\dots,0,1)$. Under such a choice of the fundamental domain the formula  (\ref{tauab3}) needs to be multiplied by an extra exponential factor (see (3.24), (3.25) of \cite{JDG} with ${\bf p}=-e_g$ and ${\bf q}=0$):
\begin{align}
&\tau(C,v)=\exp(-4\pi \sqrt{-1} \Omega_{gg})\nonumber\\
&\times c^{16}(x)\left(\frac{v(x)\prod_{i=1}^n E(x,p_i)}{\prod_{i=1}^{2g-2+n} E(x,z_i)} \right)^{8(g-1)}
\left(\f{\prod_{i<j}E(z_i,z_j)\prod_{i<j} E(p_i,p_j)}{\prod_{i=1}^n \prod_{j=1}^{ 2g-2+n}E(p_i,z_j) }\right)^{4}\,.
 \label{newtau}
\end{align}

For the following facts we refer to p.13 of \cite{Fay92}.
In the limit $t\to 0$ all entries of the period matrix remain finite 
except the entry $\Omega_{gg}$ which behaves as 
\begin{equation}
\Omega_{gg}= \frac{1}{2\pi \sqrt{-1}}\log t +O(t)
\label{asOm}
\end{equation}
Assuming that in the limit $t\to 0$ the point $x$ stays far from the nodes $a$ and $b$ (i.e. outside of the {\em ``plumbing zone''}), all the components of the vector of Riemann constants $K^x$ remain finite except $K^x_g$ that has the  asymptotics 
\begin{equation}
K^x_{g}= \frac{1}{4\pi \sqrt{-1}}\log t +O(1)\;.
\label{asK}
\end{equation}
At the same time, $c(x)$ and the prime form $E(x,y)$ have finite limits assuming again that $x$ and $y$, as well as the path 
connecting them 
remain outside of the plumbing zone. 
The vector $K^x-\Omega e_g$ in (\ref{newchoice}) also remains finite as $t\to 0$. Therefore, 
all the paths connecting the point $x$ with the points of divisor $(v)$ can be drawn outside of the plumbing zone. Thus, all the factors in the expression (\ref{newtau}) remain finite as $t\to 0$, except the exponential factor which behaves like $t^2$ due to (\ref{asOm}).
\end{proof}

{\bf Asymptotics of $\tau$ near ${\rm pr}^*\delta_{j,I}$.} Two different  cases are covered by the following two lemmas.

\begin{lemma}
Let $C$ degenerate to a reducible curve with components $C_1$ of genus $j$ and $C_2$ of genus $g-j$ such that the poles $p_{i_1},\ldots, p_{i_k}$ belong to $\CC_1$ and the poles $p_{i_{k+1}},\ldots,p_{i_n}$ belong to $C_2$. Let also 
assume that
\[
\sum_{l=1}^k r_{i_l}=\sum_{l=k+1}^n r_{i_l}= 0 \;.
\]
Then in terms of any transverse local coordinate $t$ near ${\rm pr}^*\delta_{j,I}$ in $\ol{\Hcal}_{g,n}[{\bm r}]$ the asymptotics of $\tau$ is given by 
\begin{align}\label{red1}
\tau(C,v)=const\cdot t^3+o(t^3)\quad {\rm as}\;t\to 0
\end{align}
with some $const\neq 0$.
\end{lemma}

\begin{proof}
In this case two zeros of $v$, say, $z_1$ and $z_2$, 
coalesce and disappear at the node that becomes a regular point for the both of components $C_1$ and $C_2$.
Then $t=\left(\int_{z_1}^{z_2} v\right)^2$ can be taken as a  transversal local coordinate near ${\rm pr}^*\delta_{j,I}$.

More precisely, the degeneration picture is qualitatively described by the standard fixture $\{zw=t\big| |z|,|w|,|t|<1\}$ 
fibered over the disc $\{|t|<1\}$. The fiber over $t\neq 0$ is  the annulus $\{|t|<|z|<1\}$ (identified with $\{|t|<|w|<1\}$
by the map $w=t/z$), while the fiber over $t=0$ is the union of two discs $\{z=0\}$ and $\{w=0\}$ intersecting at the origin. Up to rescaling, 
the 1-form $v$ locally looks like $v\sim(1-t/z^2)dz=(1-t/w^2)dw$. The two coalescing zeros are, approximately, $z_{1,2}=\pm t^{1/2}$,
and $v(z)\to 1$ as $t\to 0$. 
This local behavior of $v$ is demonstrated by the model case of 
$\ol{\Hcal}_{0,4}$, where the differential $v$ is given by the formula 
(\ref{vdz}). Near the boundary point, say, $p=0$,  the plumbing parameter $t$ can be identified with $p$, and zeros of $v$ as functions of $t$ can be explicitly found  from 
(\ref{vdz})).

Now we will use the alternative expression for the tau function given by (\ref{tauKoko}).
Assuming that $x$ on $C_1$ and $y$ on $C_2$ stay  outside of the plumbing zone as $t\to 0$, we can analyze the asymptotics of all terms in (\ref{tauKoko}). Namely, we have
\begin{equation}
c(x)\sim t^{-(g-j)(g-j-1)/2}\;,\qquad 
E(x,y)\sim t^{-1/2}\;,\qquad
\sigma(x,y)\sim t^{(g-2j)/2}
\la{asCE}
\end{equation}
(see \cite{Fay92}, p. 14). If both $x$ and $y$ stay on $C_1$ or $C_2$ outside of the plumbing zone, then  $E(x,y)$ and $\sigma(x,y)$  remain finite as $t\to 0$. 

The asymptotics of $E(x,z_1)$, $E(x,z_2)$,  $\sigma(x,z_1)$ and  $\sigma(x,z_2)$ as $t\to 0$ are less trivial since both $z_1$ and $z_2$ tend to the nodal point. These asymptotics were computed in \cite{Koko}, Section~2.3. Namely,
$$
E(x,z_i)\sim t^{-1/8}\;,\quad i=1,2,
$$
see (2.35) in \cite{Koko},
and
$$
\sigma(x,z_i)\sim t^{(3g-4j)/8}\;,\quad i=1,2,
$$
where $j$ is the genus of $C_1$, see (2.40) in \cite{Koko}.

Now assume that the poles $p_i,\dots, p_k$ stay on $C_1$ and the poles $p_{k+1},\dots,p_n$ stay on $C_2$ as $t\to 0$;
the zeros $z_{1,2}$ tend to the node while the zeros $z_3,\dots,z_{2j+k}$ stay on $C_2$ and the remaining $2(g-j)-2+(n-k)$ zeros stay on $C_2$. We use the above asymptotic formulas to compute the total power of $t$ in the asymptotics of $\tau(C,v)$ 
as $t\to 0$:
\begin{itemize}
\item the term $c^{16}(x)$ gives the power $-8(g-j)(g-j-1)$ of $t$; 
\item the product of $\sigma^{4 k_i}(x,x_i)$ when $x_i$ remains on 
$C_2$ gives the power $4(g-2j)(g-j-1)$ of $t$;
\item the product $(\sigma(x,z_1)\sigma(x,z_2))^4$ gives the power $3(g-j)-j$ of $t$;
\item the product of prime forms $E^{4(1-g)k_j}(x_i,x)$ for $x_i$'s staying on $C_2$ gives the power 
$2(2(g-j)-2)(g+2j-1)$ of $t$;
\item the product $(E(x,z_1)E(x,z_2))^{4(1-g)}$ gives the power $g-1$ of $t$.
\end{itemize}
Summing up all of these powers we get $3$, as claimed in (\ref{red1}).
\end{proof} 

\begin{lemma}
Let $C$ degenerate to a reducible curve with components $C_1$ of genus $j$ and $C_2$ of genus $g-j$ such that the poles $p_{i_1},\ldots p_{i_k}$ belong to $C_1$ and the poles $p_{i_{k+1}},\ldots,p_{i_n}$ belong to $C_2$. Let also 
assume that
\[
\sum_{l=1}^k r_{i_l}=-\sum_{l=k+1}^n r_{i_l}\neq 0 \;.
\]
Then in terms of any transverse local coordinate $t$ near ${\rm pr}^*\delta_{j,I}$ in $\ol{\Hcal}_{g,n}[{\bm r}]$ the asymptotics of $\tau$ is given by 
\begin{align}\label{red2}
\tau(C,v)=const\cdot t^2+o(t^2)\quad {\rm as}\;t\to 0
\end{align}
with some $const\neq 0$.
\end{lemma}

\begin{proof}
Contracting a separating loop on $C$ we get a nodal curve and a meromorphic 1-form on it with simple poles at the preimages of the node $a\in C_1$ and $b\in C_2$ with residues $r_0$ and $-r_0$, respectively, where  $r_0=-\sum_{l=1}^k r_{i_l}\neq 0$. As before, we assume that the poles $p_{i_1},\ldots p_{i_{k}}$ belong to $C_1$ together with $2j-1+k$ zeros $z_i$. The remaining poles $p_{i_{k+1}},\ldots, p_{i_n} $ belong to $C_2$ together with $2(g-j)-1+n-k$ zeros $z_i$; all zeros and poles (except for $a,b$) stay away from the node.

Like in the previous lemma, the local degeneration picture is described by the standard fixture $\{zw=t\big| |z|,|w|,|t|<1\}$ 
fibered over the disc $\{|t|<1\}$. We remind that the fiber over $t\neq 0$ is  the annulus $\{|t|<|z|<1\}$ (identified with $\{|t|<|w|<1\}$
by the map $w=t/z$), while the fiber over $t=0$ is the union of two discs $\{z=0\}$ and $\{w=0\}$ intersecting at the origin.
Then the 1-form $v$ asymptotically behaves like $v\sim r_0\frac{dz}{z}=-r_0\frac{dw}{w}$ as $t\to 0$ and develops
simple poles with opposite residues at the points $a\in C_1$ and $b\in C_2$. Again,
this can be seen from the formula \eqref{vdz}.

Following \cite{Fay92}, p.~14, Case II, we compute the asymptotics of all terms in (\ref{tauKoko}) as $t\to 0$. 
Assuming that $x\in C_1$ and using the asymptotics (\ref{asCE}) we have
$$
c(x)^{16}\sim t^{-8((g-j)^2+g-j)}
$$
In the  product of $\sigma^{k_i}(x,x_i)$ only the terms with $x_i\in \CC_2$ contribute to non-trivial powers of $t$.
Thus
$$
\prod_{i=1}^N \sigma^{4k_i}(x,x_i)= \frac{\prod_{i=1}^{2(g-j)-1+n-k}\sigma^4(x,z_i)}{\prod_{i=1}^{n-k}\sigma^4(x,p_i)}\sim
t^{2(g-2j)(2(g-j)-1)}
$$
and, finally,
$$
\prod_{i=1}^N  E^{k_i} (x,x_i)=\frac{\prod_{i=1}^{2(g-j)-1+n-k} E(x,z_i)}{\prod_{i=1}^{n-k}E(x,p_i)}\sim t^{1/2(2(g-j)-1)}\,.
$$
Substituting these asymptotics into (\ref{tauKoko}) we get (\ref{red2}).
\end{proof}

Combining Proposition \ref{seclL} with asymptotic formulas  \eqref{deg}, \eqref{irr}, \eqref{red1} and \eqref{red2} we get 
\begin{theorem}\label{lapsi3}
Fix $r_1,\ldots,r_n\in\mathbb{C}$ such that $\sum_{i=1}^nr_i=0$ and $r_i\neq 0,\;i=1,\ldots,n$.
Then the following relation holds in the Picard group ${\rm Pic}(\ol{\Mcal}_{g,n})\otimes{\mathbb Q}$ 
\begin{align}\label{main}
\delta_{deg}[{\bm r}]=24\lambda+2\sum_{i=1}^n\psi_i-2\delta_{irr}-\frac{1}{2}\,\sum_{j=0}^g\,\sum_{I\subset [n]}\alpha_I\,\delta_{j,I}\,,
\end{align}
where $I=\{{i_1},\ldots, {i_k}\}$ runs over all subsets of $[n]=\{1,\ldots,n\}$, and
\begin{align}
\alpha_I=\begin{cases}2, \;{\rm if}\; \sum_{j=1}^k r_{i_j}\neq 0\,,\\
3,\;{\rm if}\; \sum_{j=1}^k r_{i_j}=0\,.\end{cases}
\end{align}
\end{theorem}

\begin{remark}
The condition $r_i\neq 0,\; i=1,\ldots, n,$ is crucial for our approach. When $r_i\to 0$, the tau function \eqref{tauab3} 
acquires an essential singularity that makes its asymptotic analysis complicated, if at all doable.
\end{remark}

Let us illustrate Theorem 1 on two simple examples. 

\smallskip
\emph{Example 1.} Let $g=0$ and $n=4$. Then $\ol{\Mcal}_{0,4}=\mathbb{C}P^1$, its boundary divisor is $\delta=\{0,1,\infty\}$, and in this case $\lambda=0,\;\delta_{irr}=0$. Moreover, combining Mumford's formula
\begin{align}\label{Mu}
\kappa_1=12\lambda + \sum_{i=1}^n \psi_i - \delta
\end{align}
(cf., e.g., \cite{ACG}, Chapter XIII, Eq.~(7.7)) with the formula
\[
\kappa_1=\frac{1}{2(n-1)}\sum_{k=2}^{n-2}\sum_{|I|=k}(k-1)(n-k-1)\,\delta_{0,I}
\]
from \cite{Z}, Section 5, we see that in our case
$\sum_{i=1}^4 \psi_i =\frac{4}{3}\, \delta$; here $\kappa_1$ is the first Mumford's class (also known as the Weil-Petersson class).

A general meromorphic 1-form is then given by the formula
\begin{equation}
v(z)dz =\left(\frac{r_1}{z-p}+\frac{r_2}{z}+\frac{r_3}{z-1}\right)dz\,,\quad z\in\mathbb{C}\,, 
\label{vdz}
\end{equation}
where $p,\,0,\,1$ and $\infty$ are the simple poles of $v$ with residues
$r_1,r_2,r_3$ and $r_4=-r_1-r_2-r_3$ respectively. There are three distinct cases to consider.

{\it Case 1}. This is generic situation, when $r_1+r_2\neq 0,\,r_1+r_3\neq 0,\,r_2+r_3\neq 0$. The divisor $\delta_{deg}$ intersects each space $\ol{\Hcal}_{0,4}[{\bm r}]$ at precisely two points with coordinates 
\[
p_{1,2}=-\frac{r_1r_3+r_2r_4\pm 2\sqrt{r_1r_2r_3r_4}}{(r_2+r_3)^2}
\]
corresponding to the 1-forms with a double zero,
so that in this case $\delta_{deg}[{\bm r}]=\frac{2}{3}\,\delta$. Since here 
$\sum_{I\subset [4]}\alpha_I\,\delta_{0,I}=6\cdot\frac{1}{3}\,\delta=2\,\delta$, this agrees with \eqref{main}.

{\it Case 2.} Here $r_1+r_2 \neq 0,\;r_1+r_3\neq 0,\;r_2+r_3 = 0$. The divisor $\delta_{deg}$ intersects each space 
$\ol{\Hcal}_{0,4}[{\bm r}]$ at precisely one point with coordinate 
\[
p=\frac{(r_1+r_2)^2}{4r_1r_2}\,, 
\]
corresponding to the 1-form with a double zero,
so that in this case $\delta_{deg}[{\bm r}]=\frac{1}{3}\,\delta$. Since here 
$\sum_{I\subset [4]}\alpha_I\,\delta_{0,I}=(4+3)\cdot\frac{1}{3}\,\delta=\frac{7}{3}\,\delta$, this also agrees with \eqref{main}.

{\it Case 3.} Here $r_1+r_2 = 0,\,r_1+r_3\neq 0,\,r_2+r_3 = 0$. The intersection of $\ol{\Hcal}_{0,4}[{\bm r}]$  with
$\delta_{deg}$ is trivial, so that in this case $\delta_{deg}[{\bm r}]=0$. Since here 
$\sum_{I\subset [4]}\alpha_I\,\delta_{0,I}=(2+6)\cdot\frac{1}{3}\,\delta=\frac{8}{3}\,\delta$, this again agrees with \eqref{main}.

\smallskip
\emph{Example 2.} Take $g=1,\;n=2,\;r_1=1,\;r_2=-1$,  and compute the pushforward class $\delta_{deg}[\pm 1]$ in the Picard group
${\rm Pic}(\overline{\mathcal{M}}_{1,2})\otimes\mathbb{Q}$ in two different ways.

On the one hand, formula \eqref{main} adapted for this case reads
\begin{align}\label{ex2}
\delta_{deg}[{\pm 1}]=24\lambda+2(\psi_1+\psi_2)-2\delta_{0}-3\delta_1,
\end{align}
where  $\delta_0=\frac{1}{2}[\overline{\mathcal{M}}_{0,4}]$ and $\delta_1=[\overline{\mathcal{M}}_{1,1}]$
are the classes of the boundary divisors of $\overline{\mathcal{M}}_{1,2}$.

On the other hand, we can modify to our case (we briefly sketch it below) an approach of D.~Zvonkine \cite{Zvo} originally developed for the case of holomorphic 1-forms (some details can be found in \cite{S}, Section 1.6, An example) which  leads to the same result.

Namely, denote by 
$\overline{\mathcal{H}}_{1,2}^{1}[\pm1]$ the moduli space of 
meromorphic 1-forms on elliptic curves with two simple poles of residues $\pm 1$ and one marked point, say, $x\in C,\; x\neq p_1, p_2$. 
Let $\mathcal{L}$ be the tautological line
bundle on $\overline{\mathcal{H}}_{1,2}^{1}[\pm1]$, where $\mathcal{L}\left|_{(C,v,x)}\right.=T^*_x C$.
It comes with a canonical section that assigns to the 1-form $v$ the cotangent vector $v(x)\in\mathcal{L}$. 
This section vanishes precisely when $v$ has a zero at $x$, and we denote the closure of its zero 
set by $\Delta_1$. Clearly, $[\Delta_1]=\psi$, where $\psi=c_1(\mathcal{L})$.

Let us turn our attention to the locus $\Delta_1$. The restriction of $\mathcal{L}^{\otimes 2}$ to $\Delta_1$ has a section that vanishes when a 1-form has a second order zero at the marked point $x$. Denote the corresponding sub-locus by
$\Delta_2$; clearly, $[\Delta_2]=\psi\cdot (2\psi)$. The locus $\Delta_2$ consists of components of three types:
\begin{enumerate}[label=(\arabic*)]
\item The closure $\Delta_2^{(1)}$ of the locus where the marked point is a zero of order exactly 2.
\item The locus $\Delta_2^{(2)}$, where the marked point  belongs to an elliptic tail attached to the rest of the curve by a single node.
\item The locus $\Delta_2^{(3)}$, where the marked point belongs to a rational curve connecting two otherwise disconnected components of the curve.
\end{enumerate}
Push the equality $[\Delta_2^{(1)}]+[\Delta_2^{(2)}]+[\Delta_2^{(3)}]=2\psi^2$ forward using the natural forgetful map
${\rm pr}:\overline{\mathcal{H}}_{1,2}^{1}[\pm1]\rightarrow\overline{\mathcal{M}}_{1,2}$.  We get
\[
{\rm pr}_*([\Delta_2^{(1)}]) =\delta_{deg}[\pm1],\;\; {\rm pr}_*([\Delta_2^{(2)}])=0,\;\;{\rm pr}_*([\Delta_2^{(3)}])=\delta_1,\;\;2\,{\rm pr}_*(\psi^2)=2\kappa_1
\]
that yields
\begin{align}
\delta_{deg}[\pm1]=2\kappa_1-\delta_1.
\end{align}
Substituting Mumford's formula \eqref{Mu} into the last equality we get formula \eqref{ex2}.

Moreover, Mumford's computation of the Chow ring of $\overline{\mathcal{M}}_{2}$ (cf. \cite{Mu}, Part III) implies, in particular, that
\[
\lambda=\frac{1}{12}\,\delta_0,\quad\psi_{1,2}=\frac{1}{12}\,\delta_0+\delta_1,\quad\kappa_1=\frac{1}{6}\,\delta_0+\delta_1,
\]
and, therefore, $\delta_{deg}[\pm 1]=\frac{1}{3}\,\delta_0+\delta_1$ in the Picard group 
${\rm Pic}(\overline{\mathcal{M}}_{1,2})\otimes\mathbb{Q}$.

\begin{remark}
After this paper was completed, A.~Sauvaget informed us that the methods of his paper \cite{S} would, in principle, allow one to express the divisor class $\delta_{deg}$ in terms of the standard generators of the rational Picard group of the moduli space 
$\overline{\mathcal{H}}_{g,n}$.
\end{remark}

\begin{remark}
The tau function on moduli spaces of holomorphic differentials (when $n=0$) gives the holomorphic factorization formula for the determinant of Laplacian on a Riemann surface with flat conical metric $|v|^2$ (see \cite{JDG}).
For the meromorphic case $n>0$ the volume of a Riemann surface in flat metric $|v|^2$ is infinite; this makes the 
spectral theory more complicated. Nevertheless, the determinant of Laplacian can be defined on the so-called Mandelstam diagrams (see \cite{GidWol}) that correspond to the space of pairs $(C,v)$, where all periods of the differential $v$ are imaginary. On this real slice of the moduli space the determinant of Laplacian was defined and computed in terms of the tau function (\ref{tauab3}) in \cite{KokoCMP}, Corollary 1.
\end{remark}

{\bf Acknowledgements.} We are grateful to the anonimous referee for carefully reading the manuscript and suggesting several
improvements. We also thank A.~Kokotov, A.~Zorich and D.~Zvonkine for useful discussions.

\appendix
\section{Transformation of tau function under a change of canonical basis of cycles}
\label{symplectic}

Here we prove that the expression (\ref{defTcal}) is a section of $24$th power of the Hodge line bundle, i.~e. it transforms properly under symplectic changes of canonical bases of cycles on $C$.
Let a canonical basis of cycles $\{a_j,b_j\}_{j=1}^g$ transform to a new basis
$\{\hat{a}_j,\hat{b}_j\}_{j=1}^g$ by a matrix $G\in Sp(2g,\Z)$:
\begin{equation}
\left(\begin{array}{c} \hat{ b} \\ \hat{a} \end{array}\right)=
G \left(\begin{array}{c}  b \\ a \end{array} \right)\;,
\label{symptrans}
\end{equation}
where
$$
G=\left(\begin{array}{cc}  A & B \\ C &D\end{array}\right)\;.
$$
Then the column vector of $v$ of (normalized) Abelian differentials $v_1,\dots,v_g$, the period matrix $\Omega$, and the prime form $E(x,y)$ transform as follows (\cite{Fay92}, pp.11-12):
\begin{align}
 &\hat{v}= {\rm ^t}\!(C\Omega+D)^{-1} v\;,
 \label{transv}\\
 &\widehat{\Omega}=(A\Omega+B)(C\Omega+D)^{-1}\;,\nonumber\\
&\widehat{E}(x,y)= E(x,y) e^{\pi\sqrt{-1} \langle Q \Acal|_x^y, \, \Acal|_x^y \rangle}\;,
\label{transE}
\end{align}
where 
$$Q=(C\Omega +D)^{-1}C\;.$$

Transformation formulas for the vector of Riemann constants $K^x$ and the multi-differential $c(x)$ also involve two vectors  $\alpha_0,\beta_0\in \frac{1}{2} \Z^g$ defined  by the requirement that 
$$
\alpha_0-\frac{1}{2} (C{\rm ^t}\! D)_0\in \Z^g\;,\qquad
\beta_0-\frac{1}{2} (A{\rm ^t}\! B)_0\in \Z^g\;,
$$
see Lemma 1.5 of \cite{Fay92};
here the subscript $0$ denotes the vector consisting of diagonal entries of a square matrix. The above conditions fix vectors $\alpha_0$ and $\beta_0$ up to a period of the Jacobian; this period if determined by the transformation law of the vector of the Riemann constants: 

\begin{equation}
\widehat{K}^x= M^{-1} K^x +\widehat{\Omega} \alpha_0 +\beta_0
\label{transK}\end{equation}
with 
$$M= {\rm ^t}\!(C\Omega+D)\;.$$

Finally, the transformation of the differential $c(x)$ looks as follows (see (1.23) of \cite{Fay92}):

\begin{align}
&\hat{c}(x)=\epsilon \, ({\rm det} M)^{3/2}\, c(x)\nonumber\\
&\times\exp(\pi\sqrt{-1}\langle K^x , \;Q K^x\rangle -\pi\sqrt{-1}\langle \alpha_0,\, \widehat{\Omega}\alpha_0\rangle -2\pi\sqrt{-1} \langle \alpha_0,  M^{-1} K^x\rangle)\;,
\label{transc}\end{align}
where $\epsilon^8=1$.

Now we will need the following
\begin{lemma}
Under symplectic transformation (\ref{symptrans}) of the canonical basis of cycles the tau function (\ref{tauab3}) transforms as follows:
\begin{equation}
\hat{\tau}= ({\rm det} M)^{24} \tau
\label{transtau}\end{equation}
\end{lemma}
\begin{proof}
The proof is obtained by substituting the transformation formulas (\ref{transE}), (\ref{transK}) and (\ref{transc}) into the formula for the 
tau function (\ref{taugen});
recall that it holds when the fundamental polygon is chosen such that $\Acal_x((v))+2 K^x =0$. 

Under an arbitrary choice of the fundamental polygon this  relation  holds up to an integer combination of the Jacobian lattice vectors,
i.e. there exist two vectors ${\bf r}, {\bf q}\in \Z^g$ such that
\begin{equation}
\Acal_x((v))+2 K^x+\Omega {\bf p}+ {\bf q} =0
\label{rq}
\end{equation}
If the divisor $(v)$ satisfies (\ref{rq}), then the expression (\ref{taugen}) gets an extra exponential factor given by
\begin{equation}
\label{expfac}
\exp\{-4\pi\sqrt{-1}\langle {\bf r},\Omega {\bf p}\rangle - 16 \pi\sqrt{-1}\langle {\bf p}, K^x\rangle\}
\end{equation}
(see (3.24) and (3.25) of \cite{JDG}).

Let us now apply the symplectic transformation (\ref{symptrans}) to the right hand side of (\ref{taugen}) and denote the result by $\hat{\tau}$. 
To write down the formula for $\hat{\tau}$ in terms of $\hat{E}$ and $\hat{c}$ we observe that the Abel map $\hat{\Acal}_x$ of divisor $(v)$ can be computed using the relation $\Acal_x((v))+2 K^x=0$ combined with (\ref{transv}) and (\ref{transK}):
\begin{equation}
\label{newrel}
\hat{\Acal}_x((v))+2 \hat{K}^x-2\widehat{\Omega} \alpha_0 -2\beta_0=0\;,
\end{equation}
i.e. in (\ref{rq}) we get ${\bf p}=-2\alpha_0$, ${\bf q}=-2\beta_0$.

Therefore, according to (\ref{expfac}), 
\begin{align}
\hat{\tau}&=\hat{c}^{16}(x)\left(\frac{v(x)}{\prod_{i=1}^N \widehat{E}^{k_i}(x,x_i)}\right)^{8(g-1)} \prod_{i<j} \widehat{E}^{4 k_i k_j}(x_i, x_j)\nonumber\\
\label{tauh}
&\times \exp(-16\pi\sqrt{-1}\langle \alpha_0,\widehat{\Omega} \alpha_0\rangle + 32 \pi\sqrt{-1}\langle \alpha_0, \widehat{K}^x\rangle)\;.
\end{align}

Due to (\ref{transK}) the exponential factor in (\ref{tauh}) can be rewritten as
\begin{equation}
\label{exp1}
\exp(16\pi\sqrt{-1}\langle \alpha_0,\widehat{\Omega} \alpha_0\rangle + 32 \pi\sqrt{-1}\langle \alpha_0,  M^{-1} {K}^x\rangle)\;,
\end{equation}
where we used the equality $\exp(32\pi\sqrt{-1}\langle \alpha_0,\,\beta_0\rangle)=1$ that holds because $\alpha_0$ and $\beta_0$ are vectors with half-integer components.

Compute now $\log \frac{\hat{\tau}}{\tau}$ using (\ref{transE}), (\ref{transc}). Recall that $\sum_{i=1}^N k_i=2g-2$, and put $u_i=  \Acal_x (x_i)$. Then the contribution of the product of exponents in the denominator of $\log \frac{\hat{\tau}}{\tau}$
equals to
\begin{equation}
8\pi\sqrt{-1}(1-g) \sum_{i=1}^N  k_i \langle Qu_i, u_i\rangle\;,
\label{prime1}\end{equation}
while the contribution of the prime forms in the numerator of $\hat{\tau}$ and $\tau$ is
\begin{equation}
4\pi\sqrt{-1}\sum_{i<j} k_i k_j  \langle Q(u_i-u_j),\, u_i-u_j\rangle\;.
\label{prime2}
\end{equation}
The sum of two sums (\ref{prime1}) and  (\ref{prime2}) equals to
\begin{equation}
-4\pi\sqrt{-1}\langle Q \sum_{i=1}^N  k_i u_i,\,  \sum_{j=1}^N  k_j  u_j\rangle=-16 \pi\sqrt{-1}\langle Q K^x,  K^x\rangle
\label{prime12}
\end{equation}
 (we use that $\sum_{i=1}^N k_i = 2g-2$ here).

Moreover, due to (\ref{transc}), 
\begin{align}
16\log \frac{\hat{c}(x)}{c(x)}&
=24 \log  {\rm det} (C\Omega+D) +16 \pi\sqrt{-1}\langle Q K^x,  K^x\rangle \nonumber\\
&- 16\pi\sqrt{-1}\langle \alpha_0,\widehat{\Omega} \alpha_0\rangle-32 \pi\sqrt{-1}\langle \alpha_0,  M^{-1} {K}^x\rangle\;.
\label{cchat}
\end{align}

Notice that the exponential factor \eqref{exp1} in the expression for $\log (\hat{\tau}/\tau)$ is canceled by the last two terms in (\ref{cchat}),
while the second term in the right hand side of (\ref{cchat}) cancels against (\ref{prime12}). Finally, we arrive at  (\ref{transtau}).
\end{proof}

\end{document}